\documentclass[12pt]{article}  
\usepackage{amsmath}
\usepackage{theorem}                 
\usepackage{amssymb}
\usepackage{xypic}

\newtheorem{proposition}{Proposition}

\newenvironment{definition}
{\smallskip\noindent{\bf Definition\/}:}{\smallskip\par}
\newenvironment{statement}
{\smallskip\noindent{\bf Statement\/}.}{\smallskip\par}
\newenvironment{example}
{\smallskip\noindent{\bf Example\/}.}{\smallskip\par}
\newenvironment{remark}
{\smallskip\noindent{\bf Remark\/}.}{\smallskip\par}

\newenvironment{proof}
{\noindent{\it Proof\/}.}{{ \hfill $\Box$}\smallskip\par}

\newcommand{\CC}{{\Bbb C}}

\newcommand{\TT}{{\Bbb T}}
\newcommand{\calO}{{\cal O}}

\newcommand{\calA}{{\cal A}}

\newcommand{\calF}{{\cal F}}

\newcommand{\eps}{\varepsilon}

\newcommand{\ind}{{\rm ind}}
\newcommand{\rk}{{\rm rk\,}}
\newcommand{\Mmn}{M_{m,n}}
\newcommand{\indPH}{{\rm ind}_{\rm PH}^i \,}
\newcommand{\indPHN}{{\rm ind}_{\rm PHN}\,}
\newcommand{\indPHk}{{\rm ind}_{\rm PH}^k \,}
\newcommand{\indrad}{{\rm ind}_{\rm rad}\,}
\newcommand{\indhom}{{\rm ind}_{\rm hom}\,}
\newcommand{\indGMvS}{{\rm ind}_{\rm GMvS}\,}
\newcommand{\indPHs}{{\rm ind}_{\rm PH}\,}

\title{On indices of 1-forms on determinantal singularities}
\author{W.~Ebeling and S.~M.~Gusein-Zade
\thanks{Partially supported by the DFG (Eb 102/5--1), INTAS-05-7805, RFBR--07--01--00593,
NWO-RFBR 047.011.2004.026.
Keywords: determinantal variety, 1-form, index.
AMS Math. Subject Classification: 14B05, 14M12, 58A10.
}
}
\date{}

\begin{document}

\maketitle

\begin{abstract}
We consider 1-forms on, so called, essentially isolated determinantal singularities (a natural generalization of isolated ones), show how to define analogues of the Poincar\'e--Hopf index for them, and describe relations between these indices and the radial index. For isolated determinantal singularities, we discuss properties of the homological index of a holomorphic 1-form and its relation with the Poincar\'e--Hopf index.
\end{abstract}

\section*{Introduction}

There are several generalizations of the notion of the index of a singular point (zero) of a vector field or of a 1-form on a smooth manifold (real or complex analytic) to vector fields or 1-forms on singular varieties. Some of them (the radial index, the Euler obstruction) are defined for arbitrary varieties. The homological index is defined for holomorphic vector fields or 1-forms on isolated singularities. The GSV (G\'omez-Mont--Seade--Verjovsky) index or PH (Poincar\'e--Hopf) index is defined for isolated complete intersection singularities: {ICIS}. For definitions and properties of these indices see \cite{EGSch} and the references there. This paper emerged from an attempt to generalize the notion(s) of the GSV- or PH-indices from complete intersections to more general varieties. It is not clear how this can be done in the general setting. 
The GSV-index of a vector field or of a 1-form on an ICIS can be considered as a (duely understood) localization of the top Chern-Fulton (or Chern-Fulton-Johnson) class of a variety with isolated complete intersection singularities. It seems that in general such a localization is not well-defined (at least as an appropriate index of a vector field or of a 1-form).

A more general class than complete intersection singularities is the class of determinantal singularities which are defined by vanishing of all the minors of a certain size of an $m \times n$-matrix. We consider so called essentially isolated determinantal singularities (EIDS) as a natural generalization of isolated ones. For a 1-form on an EIDS, we define several analogues of the Poincar\'e--Hopf index using natural resolutions of deformations of the EIDS. We determine relations between these indices and the radial index. For isolated determinantal singularities, we discuss properties of the homological index of a holomorphic 1-form and its relation with the Poincar\'e--Hopf index.

\section{Determinantal singularities}
A determinantal variety $X$ (of type $(m,n,t)$) in an open domain $U\subset\CC^N$ is a variety of dimension $d:=N-(n-t+1)(m-t+1) $ defined by the condition $\rk F(x) < t$ where $t\le \min(m, n)$, $F(x)=\left(f_{ij}(x)\right)$ is an $m\times n$-matrix ($i=1,\ldots, m$, $j=1,\ldots, n$), whose entries $f_{ij}(x)$  are complex analytic functions on $U$. In other words, $X$ is defined by the equations $m^t_{IJ}(x)=0$ for all subsets $I\subset\{1, \ldots, m\}$, $J\subset\{1, \ldots, n\}$ with $t$ elements, where $m^t_{IJ}(x)$ is the corresponding $t\times t$-minor of the matrix $F(x)$. 

This definition can be reformulated in the following way. Let $\Mmn \cong \CC^{mn}$ be the space of $m \times n$-matrices with complex entries and let $\Mmn^t$ be the subset of $\Mmn$ consisting of matrices of rank less than $t$, i.e.\ of matrices all $t \times t$-minors of which vanish. The variety $\Mmn^t$ has codimension $(m-t+1)(n-t+1)$ in $M_{m,n}$. It is singular. The singular locus of $\Mmn^t$ coincides with $\Mmn^{t-1}$. The singular locus of the latter one coincides with $\Mmn^{t-2}$, etc.\ (see, e.g., \cite{ACGH}). 
The representation of the variety $\Mmn^{t}$ as the union of $\Mmn^{i}\setminus\Mmn^{i-1}$, $i=1, \ldots, t$, is a Whitney stratification of $\Mmn^{t}$.
The matrix $F(x)=(f_{ij}(x))$, $x \in \CC^N$, determines a map $F: U \to \Mmn$ and the determinantal variety $X$ is the preimage $F^{-1}(\Mmn^t)$ of the variety $\Mmn^t$ (subject to the condition that ${\rm codim}\, X = {\rm codim}\, \Mmn^t$).
For $1\le i\le t$, let $X_i=F^{-1}(\Mmn^{i})$.

The image of a generic map $F: U \to \Mmn$ may intersect the varieties $\Mmn^{i}$ for $i < t$ (therefore the corresponding determinantal variety $F^{-1}(\Mmn^t)$ will have singularities). However, a generic map $F$ intersects the strata $\Mmn^{i} \setminus \Mmn^{i-1}$ of the variety $\Mmn^t$ transversally. This means that, at the corresponding points, the determinantal variety has ''standard'' singularities whose analytic type only depends on $i= {\rm rk}\, F(x) +1$. This inspires the following definition:

\begin{definition} A point $x \in X=F^{-1}(\Mmn^t)$ is called {\it essentially non-singular} if, at the point $x$, the map $F$ is transversal to the corresponding stratum of the variety $\Mmn^t$ (i.e., to $\Mmn^{i} \setminus \Mmn^{i-1}$ where $i= {\rm rk} \, F(x)+1$). 
\end{definition}

At an essentially non-singular point $x_0\in X$ with ${\rm rk} \, F(x_0)=i-1$, 
$${\rm rk} \{dm_{IJ}^{i} \} = (m-i+1)(n-i+1)$$
($\{dm_{IJ}^{i} \}$ is the set of the differentials of all the $i \times i$-minors of the matrix $F(x)$) and, in a neighbourhood of the point $x_0$,
the subvariety $X_i \setminus X_{i-1} = \{ x \in X \, | \, {\rm rk} \, F(x)=i-1\}$ is non-singular of dimension $N-(m-i+1)(n-i+1)$. At an essentially singular point $x \in X$ one has
$${\rm rk} \{dm_{IJ}^{i} \} < (m-i+1)(n-i+1)\,.$$

\begin{definition}
A germ $(X,0)\subset(\CC^N,0)$ of a determinantal variety of type $(m,n,t)$ has an {\em isolated essentially singular point} at the origin (or is an {\it essentially isolated determinantal singularity}: {EIDS}) if it has only essentially non-singular points in a punctured neighbourhood of the origin in $X$.
\end{definition}

An example of an {EIDS} is an isolated complete intersection singularity ({ICIS}): it is an {EIDS} of type $(1,n,1)$. 

From now on we shall suppose that a determinantal singularity $(X,0)$ of type $(m,n,t)$ is defined by a map $F: (\CC^N,0) \to \Mmn$ with $F(0)=0$. This is not a restriction because if $F(0) \neq 0$ and therefore ${\rm rk} \, F(0)=s>0$, $(X,0)$ is a determinantal singularity of type $(m-s,n-s,t-s)$ defined by a map $F': (\CC^N,0) \to M_{m-s,n-s}$ with $F'(0)=0$. We also suppose that $\dim X >0$, i.e.\ $N>(m-t+1)(n-t+1)$. Let $\eps>0$ be such that, for all positive $\eps'\le \eps$, each stratum $X_i\setminus X_{i-1}$ of the variety $X$, $1\le i \le t$, is transversal to the sphere $S_{\eps'}=\partial B_{\eps'}$ of radius $\eps'$ centred at the origin in $\CC^N$. We shall suppose deformations $\widetilde X$ of the EIDS 
$(X,0)$ being defined in a neighbourhood $U$ of the ball $B_\eps$ and being such that the corresponding strata stay transversal to the sphere $S_\eps$.

An essentially isolated determinantal singularity 
$(X,0) \subset (\CC^N,0)$ of type $(m,n,t)$ (defined by a map $F: (\CC^N,0) \to (\Mmn, 0)$) has an isolated singularity at the origin if and only if $N \leq (m-t+2)(n-t+2)$.

In the sequel we shall consider deformations (in particular, smoothings) of determinantal singularities which are themselves determinantal ones, i.e., they are defined by perturbations of the map $F$ defining the singularity.

\begin{remark} Determinantal singularities (in particular isolated ones) may have deformations (even smoothings) which are not determinantal (see an example in \cite{Pi}). These smoothings may have different topology, say, different Euler characteristics.
\end{remark}

Let $(X,0) \subset (\CC^N,0)$ be an EIDS defined by a map $F: (\CC^N,0) \to (\Mmn,0)$ ($X=F^{-1}(\Mmn^t)$, $F$ is transversal to $\Mmn^{i} \setminus \Mmn^{i-1}$ at all points $x$ from a punctured neighbourhood of the origin in $\CC^N$ and for all $i\le t$).

\begin{definition} An {\em essential smoothing} $\widetilde{X}$ of the EIDS $(X,0)$ is a subvariety of a neighbourhood $U$ of the origin in $\CC^N$ defined by a perturbation $\widetilde{F}: U \to \Mmn$ of the germ $F$ transversal to all the strata $\Mmn^{i} \setminus \Mmn^{i-1}$ with $i \leq t$.
\end{definition}

A generic deformation $\widetilde{F}$ of the map $F$ defines an essential smoothing of the EIDS $(X,0)$ (according to the Thom Transversality Theorem). An essential smoothing is in general not smooth (for $N \geq (m-t+2)(n-t+2)$). Its singular locus is $\widetilde{F}^{-1}(\Mmn^{t-1})$, the singular locus of the latter one is $\widetilde{F}^{-1}(\Mmn^{t-2})$, etc. The representation of $\widetilde{X}$ as the union 
$$\widetilde{X} = \bigcup_{1 \leq i \leq t}  \widetilde{F}^{-1}(\Mmn^{i} \setminus \Mmn^{i-1})$$
is a Whitney stratification of it. An essential smoothing of an EIDS $(X,0)$ of type $(m,n,t)$ is a genuine smoothing if and only if $N < (m-t+2)(n-t+2)$.

One can say that the variety $\Mmn^t$ has three natural resolutions. 

One of them is constructed in the following way (see, e.g., \cite{ACGH}). 
Let ${\rm Gr}(k,n)$ be the Grassmann manifold of $k$-dimensional vector subspaces in $\CC^n$. 
Consider $m \times n$-matrices as linear maps from $\CC^n$ to $\CC^m$.
Let $Y_1$
be the subvariety of the product $M_{m,n}\times {\rm Gr}(n-t+1,n)$ which consists of pairs $(A, W)$ such that $A(W)=0$. The variety $Y_1$
is smooth and connected. Its projection to the first factor defines a resolution $\pi_1: Y_1 \to \Mmn^t$ of the variety $\Mmn^t$. 

If one considers $m \times n$-matrices as linear maps from $\CC^m$ to $\CC^n$, one gets in the same way another resolution $\pi_2: Y_2 \to \Mmn^t$ of the variety $\Mmn^t$. In this case one has $Y_2 \subset \Mmn \times {\rm Gr}(m-t+1,m)$.

The third natural resolution is given by the Nash transform $\widehat{\Mmn^t}$ of the variety $\Mmn^t \subset \Mmn$. The Nash transform $\widehat{Z}$ of a variety $Z \subset U \subset \CC^N$ of pure dimension $d$ is the closure in the product $Z \times {\rm Gr}(d,N)$ of the set 
$$\{ (x,W) \, | \, x \in Z_{\rm reg}, W=T_xZ_{\rm reg} \}\,,$$
where $Z_{\rm reg}$ is the set of non-singular points of the variety $Z$. The Nash transform may in general be singular. The tautological vector bundle $\TT$ over the Nash transform $\widehat{Z}$ is the pull-back of the tautological bundle (of rank $d$) over the Grassmann manifold ${\rm Gr}(d,N)$ (under the projection $Z \times {\rm Gr}(d,N) \to {\rm Gr}(d,N)$). Over the preimage of the non-singular part $Z_{\rm reg}$ of the variety $Z$ the tautological bundle is in a natural way isomorphic to the tangent bundle. However, if the Nash transform $\widehat{Z}$ is by a chance non-singular, the tautological bundle and the tangent bundle are in general different.

Let $d_{m,n}^t:=\dim \Mmn^t = mn-(m-t+1)(n-t+1)$ and let ${\rm Gr}(d_{m,n}^t,\Mmn)$ be the Grassmann manifold of $d_{m,n}^t$-dimensional subspaces of $\Mmn$. The Nash transform $\widehat{\Mmn^t}$ is the closure in $\Mmn^t \times {\rm Gr}(d_{m,n}^t,\Mmn)$ of the set
$$\{(A,T) \, | \, A \in \Mmn^t \setminus \Mmn^{t-1}=(\Mmn^t)_{\rm reg}, T=T_A\Mmn^t \}.$$
As above, let us consider matrices $A \in \Mmn^t$ as operators $A: \CC^n \to \CC^m$. The tangent space to $\Mmn^t$ at a point $A$ with ${\rm rk}\, A =t-1$ is 
$$T_A\Mmn^t=\{ B \in \Mmn \, | \, B(\ker A) \subset {\rm im}\, A\}$$
(see, e.g., \cite{ACGH}). Let $\alpha: {\rm Gr}(n-t+1,n) \times {\rm Gr}(t-1,m) \to {\rm Gr}(d_{m,n}^t,\Mmn)$ be the map which sends a pair $(W_1, W_2)$ ($W_1 \subset \CC^n$, $W_2 \subset \CC^m$, $\dim W_1=n-t+1$, $\dim W_2=t-1$) to the $d_{m,n}^t$-dimensional vector subspace 
$$\{ B \in \Mmn \, | \, B(W_1) \subset W_2 \} \subset \Mmn.$$
This map is an embedding. The closure of the set of matrices (operators) $A$ with $\ker A=W_1$, ${\rm im}\, A=W_2$ ($(W_1,W_2) \in {\rm Gr}(n-t+1,n) \times {\rm Gr}(t-1,m)$) consists of all operators $A$ with $\ker A \supset W_1$ and ${\rm im}\, A \subset W_2$. Therefore the Nash transform $\widehat{\Mmn^t}$ of $\Mmn^t$ consists of pairs $(A,W)$ such that $W=\alpha(W_1,W_2)$, $\ker A \supset W_1$, ${\rm im}\, A \subset W_2$. This means that the projection of $\widehat{\Mmn^t} \subset \Mmn^t \times {\rm Gr}(d_{m,n}^t,\Mmn)$ to the second factor is the projection of a vector bundle of rank $(t-1)^2$ over $\alpha({\rm Gr}(n-t+1,n) \times {\rm Gr}(t-1,m))$. Thus the Nash transform gives a resolution $\pi_3: Y_3:=\widehat{\Mmn^t} \to \Mmn^t$ of the variety $\Mmn^t$.

\section{Poincar\'e--Hopf indices of 1-forms on EIDS}
Let $(X,0)=F^{-1}(\Mmn^t) \subset (\CC^N,0)$ be an EIDS and let $\omega$ be a germ of a (complex) 1-form on $(\CC^N,0)$ whose restriction to $(X,0)$ has an isolated singular point (zero) at the origin. This means that the restrictions of the 1-form $\omega$ to the strata $X_i\setminus X_{i-1}$, $i\le t$,
have no zeros in a punctured neighbourhood of the origin.

If $(X,0)$ is an ICIS (i.e. an EIDS of type $(1,n,1)$), the PH-index $\ind_{\rm PH} \, \omega$ of the 1-form on it is defined as the sum of the indices of the zeros of the restriction of the 1-form $\omega$ to a smoothing $\widetilde{X}$ of the ICIS $(X,0)$ in a neighbourhood of the origin.

An essential smoothing $\widetilde{X} \subset U$ of the EIDS $(X,0)$ (in a neighbourhood $U$ of the origin in $\CC^N$) is in general not smooth. To define an analogue of the PH-index one has to construct a substitute of the tangent bundle to $\widetilde{X}$. It is possible to use one of the following two natural ways.

One possibility is to use a resolution of the variety $\widetilde{X}$ connected with one of the three resolutions of the variety $\Mmn^t$ described above. Let $\pi_i: Y_i \to \Mmn^t$ be one of the described resolutions of the determinantal variety $\Mmn^t$ and let $\overline{X}_i = Y_i \times_{\Mmn^t} \widetilde{X}$, $i=1,2,3$, be the fibre product of the spaces $Y_i$ and $\widetilde{X}$ over the variety $\Mmn^t$:
$$\diagram
 & \overline{X}_i \ar[dl]_{\Pi_i} \ar[dd] \ar[dr]  &\\
\widetilde{X} \ar[dr]_{\widetilde{F}|_{\widetilde{X}} } & & Y_i \ar[dl]^{\pi_i} \\
& \Mmn^t &
\enddiagram
$$
The map $\Pi_i : \overline{X}_i \to \widetilde{X}$ is a resolution of the variety $ \widetilde{X}$. The lifting $\omega_i:= (j \circ \Pi_i)^\ast \omega$ ($j$ is the inclusion map $\widetilde{X} \hookrightarrow U \subset \CC^N$) of the 1-form $\omega$ is a 1-form on a (non-singular) complex analytic manifold $\overline{X}_i$ without zeros outside of the preimage of a small neighbourhood of the origin. In general, the 1-form $\omega_i$ has non-isolated zeros.

\begin{definition}
The {\em Poincar\'e--Hopf index} ({\em PH-index})
$\ind_{\rm PH}^i \, \omega = \ind_{\rm PH}^i \, (\omega;X,0)$, $i=1,2,3$,
of the 1-form $\omega$ on the EIDS $(X,0) \subset (\CC^N,0)$ is the sum of the indices of the zeros of a generic perturbation $\widetilde{\omega}_i$ of the 1-form $\omega_i$ on the manifold $\overline{X}_i$ (in the preimage of a neighbourhood of the origin in $\CC^N$). 
\end{definition}

In  other words, the PH-index $\indPH \omega$ is the obstruction to extend the non-zero 1-form $\omega_i$ from the preimage (under the map $j \circ \Pi_i$) of a neighbourhood of the sphere $S_\eps= \partial B_\eps$ to the manifold $\overline{X}_i$.

The other possibility is to take into account that the space $\overline{X}_3$ of the third resolution is the Nash transform of the variety $\widetilde{X}$ and to use the Nash bundle $\widehat{\TT}$ over it instead of the tangent bundle. This brings the idea of the Euler obstruction into play. 

The 1-form $\omega$ defines a non-vanishing section $\widehat{\omega}$ of the dual bundle $\widehat{\TT}^\ast$ over the preimage of the intersection $\widetilde{X} \cap S_\eps$ of the variety $\widetilde{X}$ with the sphere $S_\eps$.

\begin{sloppypar}

\begin{definition}
The {\em Poincar\'e--Hopf-Nash index} ({\em PHN-index}) $\indPHN \omega = \indPHN (\omega;X,0)$ of the 1-form $\omega$ on the EIDS $(X,0)$ is the obstruction to extend the non-zero section $\widehat{\omega}$ of the dual Nash bundle $\widehat{\TT}^\ast$ from the preimage of the boundary $S_\eps=\partial B_\eps$ of the ball $B_\eps$ to the preimage of its interior, i.e.\ to the manifold $\overline {X}_3$, more precisely, its value (as an element of
 $H^{2d}(\Pi_3^{-1}(\widetilde{X} \cap B_\eps), \Pi_3^{-1}(\widetilde{X} \cap S_\eps))$) on the fundamental class of the pair $(\Pi_3^{-1}(\widetilde{X} \cap B_\eps), \Pi_3^{-1}(\widetilde{X} \cap S_\eps))$.
\end{definition}

\end{sloppypar}

There is another possible definition of the PHN-index which follows from the property of the Euler obstruction described in \cite[Proposition~2.3]{STV} (see also \cite[Proposition~2.1]{EGChern}).

\begin{proposition}
The PHN-index $\indPHN \omega$ of the 1-form $\omega$ on the {EIDS} $(X,0)$ is equal to the number of non-degenerate singular points of a generic deformation $\widetilde\omega$ of the 1-form $\omega$ on the non-singular part $\widetilde{X}_{\rm reg}=\widetilde{F}^{-1}(\Mmn^t \setminus \Mmn^{t-1})$ of the variety $\widetilde X$.
\end{proposition}

All the defined indices (as well as the radial one and the local Euler obstruction) satisfy the law of conservation of number (see, e.g., \cite{EGSch}).

Let $\chi(X,0):=\chi(\widetilde{X} \cap B_\eps)$ for an essential smoothing $\widetilde{X}$ of the singularity $(X,0)$. Let us recall that $(X_i,0)=F^{-1}(\Mmn^i)$ , $i=1, \ldots , t$. The variety $\widetilde{X}_i \subset \widetilde{X}$ is an essential smoothing of the EIDS $(X_i,0)$ (of type $(m,n,i)$).

Let us denote by $G_i^k$ the preimage of a point from $\Mmn^i \setminus \Mmn^{i-1}$ under the resolution $\pi_k$, $k=1,2,3$, of the variety $\Mmn^t$, i.e.
\begin{eqnarray*}
G_i^1 & = & {\rm Gr}(n-t+1,n-i+1),\\
G_i^2 & = & {\rm Gr}(m-t+1,m-i+1), \\
G_i^3 & = & {\rm Gr}(n-t+1,n-i+1) \times {\rm Gr}(m-t+1,m-i+1).
\end{eqnarray*}
One has
\begin{eqnarray*}
\chi(G_i^1) & = &  \binom{n-i+1}{t-i}, \quad \chi(G_i^2)  =  \binom{m-i+1}{t-i}, \\
\chi(G_i^3) &  = &  \binom{n-i+1}{t-i} \binom{m-i+1}{t-i}.
\end{eqnarray*}

\begin{proposition} \label{PropPH}
One has
\begin{eqnarray*}
\indPHk (\omega;X,0) & = & (-1)^{\dim X} \sum_{i=1}^t  \left( (-1)^{\dim X_i} \indrad (\omega; X_i,0)  \right. \\
& & {}  - (-1)^{\dim X_{i-1}} \indrad(\omega; X_{i-1},0) \\
 & & \left. {} + (\chi(X_i,0) - \chi(X_{i-1},0)) \right) \cdot \chi(G_i^k).
\end{eqnarray*}
where we set $X_0=\{ 0\}$ $($and therefore $\indrad (\omega; X_0,0)=1${}$)$ and $\chi(X_0,0)=0$.
\end{proposition}

\begin{remark}
1. If $N < mn$ several summands in the formula of Proposition~\ref{PropPH} vanish. \\
2. Note that 
$\dim X_i$ is equal to 
$$\max\{ 0, N-(m-i+1)(n-i+1) \}.$$
\end{remark}

\begin{proof}
There exists a 1-form $\widetilde{\omega}$ which coincides with $\omega$ in a neighbourhood of the sphere $S_\eps$ such that it is radial in a neighbourhood of the origin and, in a neighbourhood of each singular point $x_0$ on $X_i \setminus X_{i-1}$, it looks as follows. There exists a local biholomorphism $h: \CC^{\dim X_i} \times M_{m-i+1,n-i+1} \to (\CC^N, x_0)$ which sends $\CC^{\dim X_i} \times M_{m-i+1,n-i+1}^{t-i+1}$ onto $X$ (and therefore $\CC^{\dim X_i} \times \{ 0 \}$ onto $X_i$) and one has $h^\ast \widetilde{\omega} = p_1^\ast \omega' + p_2^\ast \omega''$ where $p_1$ and $p_2$ are the projections of $\CC^{\dim X_i} \times M_{m-i+1,n-i+1}$ to the first and second factor respectively, $\omega'$ is a 1-form on $(\CC^{\dim X_i},0)$ with a non-degenerate zero at the origin (and therefore $\indrad (\omega'; \CC^{\dim X_i},0)= \pm 1$) and $\omega''$ is a radial 1-form on $(M_{m-i+1,n-i+1},0)$. The sum of the indices of the zeros of the 1-form $\widetilde{\omega}$ on $X_i \setminus X_{i-1}$ is equal to 
$$\indrad (\omega; X_i,0) - (-1)^{\dim X_i -\dim X_{i-1}} \indrad (\omega;X_{i-1},0).$$
Let $\widetilde{X}= \widetilde{F}^{-1}(\Mmn^t)$ be an essential smoothing of the singularity $(X,0)$ ($\widetilde{F}$ is small generic pertubation of $F$). Since outside of a small neighbourhood $B_{\eps'}$ of the origin (such that $\widetilde{\omega}$ is radial on its boundary $S_{\eps'}$), $X$ and $\widetilde{X}$ are diffeomorphic to each other (as stratified spaces), we can suppose $\widetilde{\omega}$ to have the same singular points on $\widetilde{X} \setminus B_{\eps'}$. Changing $\widetilde{\omega}$ inside the ball $B_{\eps'}$, we can suppose that all the singular points of the 1-form $\widetilde{\omega}$ on $\widetilde{X}$ are of the type described above. The sum of the indices of the zeros of the 1-form $\widetilde{\omega}$ on $(\widetilde{X}_i \setminus \widetilde{X}_{i-1}) \cap B_{\eps'}$ is equal to
$$(-1)^{\dim \widetilde{X}_i} (\chi(\widetilde{X}_i \cap B_{\eps'}) - \chi(\widetilde{X}_{i-1} \cap B_{\eps'}))= 
(-1)^{\dim \widetilde{X}_i}(\chi(X_i,0) - \chi(X_{i-1},0))$$
(the sign $(-1)^{\dim \widetilde{X}_i}$ reflects the difference between the radial index of a complex 1-form and of its real part, see, e.g., \cite{EGGD}).  Therefore the sum of the indices of the zeros of the 1-form $\widetilde{\omega}$ on the entire stratum $\widetilde{X}_i \setminus \widetilde{X}_{i-1}$ is equal to
\begin{eqnarray*}
\lefteqn{\indrad (\omega; X_i,0) - (-1)^{\dim X_i - \dim X_{i-1}} \indrad (\omega;X_i,0)} \\
&+ & (-1)^{\dim \widetilde{X}_i} (\chi(X_i,0) - \chi(X_{i-1},0)).
\end{eqnarray*}
Applying one of the three possible resolutions of the variety $\widetilde{X}$ over a point of $\widetilde{X}_i \setminus \widetilde{X}_{i-1}$, we glue in a standard fibre $G_i^k$ (a Grassmanian or a product of two of them). The sum of the indices of the zeros of a generic perturbation of the radial 1-form on the corresponding resolution of $M_{m-i+1,n-i+1}^{t-i+1}$ (the normal slice to $\widetilde{X}_i$ in $\widetilde{X}$) is equal to $(-1)^{\dim \widetilde{X} - \dim \widetilde{X}_i} \chi(G_i^k)$. Therefore the corresponding PH-index of the 1-form is given by the formula of Proposition~\ref{PropPH}.
\end{proof}

In the formula expressing the radial index of a 1-form in terms of the Euler obstructions of the 1-form on the different strata \cite{EGGD}, there participate certain integer coefficients $n_i$. Up to sign they are equal to the reduced Euler characteristics of generic hyperplane sections of normal slices of the variety at points of different strata of a Whitney stratification. (The reduced Euler characteristic $\overline{\chi}(Z)$ of a topological space $Z$ is $\chi(Z)-1$.) For a determinantal singularity $X=F^{-1}(\Mmn^t)$ outside of the origin these slices are standard ones: a normal slice to the stratum $X_i \setminus X_{i-1}$ ($X_i=F^{-1}(\Mmn^i)$) is isomorphic to $(M_{m-i+1,n-i+1}^{t-i+1},0)$. Therefore one has the problem to compute the Euler characteristic of a generic hyperplane section of the variety $\Mmn^t$.

\begin{proposition}
Let $\ell: \Mmn \to \CC$ be a generic linear form and let $L_{m,n}^t = \Mmn^t \cap \ell^{-1}(1)$. Then, for $t \leq m \leq n$, one has
$$\overline{\chi}(L_{m,n}^t) = (-1)^t \binom{m-1}{t-1}.$$
\end{proposition}

\begin{proof}
The linear form $\ell$ can be written as $\ell(A)= {\rm tr}\, CA$ for a certain (generic) $n \times m$-matrix $C$. Since
${\rm tr}\, CA = {\rm tr}\, S_1^{-1} C S_2 S_2^{-1} A S_1$
for invertible matrices $S_1$ and $S_2$ of size $n \times n$ and $m \times m$ respectively, one can suppose that
$$C = \left( \begin{array}{c} I_m \\ 0 \end{array} \right)$$
where $I_m$ is the unit $m \times m$-matrix and $0$ is the zero matrix of size $(n-m) \times m$. Thus one has $\ell(A) = \sum_{i=1}^m a_{ii}$ for $A=(a_{ij})$.

For $I \subset \{1, \ldots , m\}$, $I \neq \emptyset$, let 
$$L_I= \{ A \in L_{m,n}^t \, | \, a_{ii} \neq 0 \mbox{ for } i \in I, a_{ii}=0 \mbox{ for } i \not\in I \}.$$
The space $L_{m,n}^t $ is the union of the spaces $L_I$ and therefore $\chi(L_{m,n}^t)=\sum\limits_I \chi(L_I)$.

Let $I$ as above be fixed and let $k = \# I$ be the number of elements of $I$. Without loss of generality we can suppose that $I=\{1, \ldots , k\}$. Let 
\begin {eqnarray*}
B_I & = & \{ (\alpha_1, \ldots , \alpha_k) \in \CC^k \, | \, \alpha_i \neq 0 , \sum \alpha_i =1 \}, \\
C_I & = & \left\{ A \in L_I \, \left| \, a_{ii}= \frac{1}{k} \mbox{ for } i=1, \ldots , k \right\} \right. .
\end{eqnarray*}
One can see that the space $L_I$ is the direct product of the spaces $B_I$ and $C_I$. The corresponding projections to $B_I$ and $C_I$ are defined by 
\begin{eqnarray*}
p_1(A) & = & (a_{11}, \ldots , a_{kk}), \\
p_2(A) & = & D\left(\frac{1}{ka_{11}}, \ldots , \frac{1}{ka_{kk}}, 1, \ldots , 1\right) A
\end{eqnarray*}
where $D(\alpha_1, \ldots , \alpha_s)$ is the diagonal $s \times s$-matrix with diagonal entries $\alpha_1, \ldots , \alpha_s$.

The inclusion-exclusion formula gives
$$\chi(B_I) = 1 - \binom{k}{1} + \binom{k}{2} \pm \ldots +(-1)^{k-1} \binom{k}{k-1} = (-1)^{k-1}.$$
To compute the Euler characteristic of the space $C_I$, let us consider the action of the group $(\CC^\ast)^n$ on $C_I$ defined by
$$(\lambda_1, \ldots , \lambda_n) \ast A = D(\lambda_1^{-1}, \ldots , \lambda_m^{-1}) A D(\lambda_1, \ldots , \lambda_n)$$
for $\lambda_i \in \CC^\ast$. The Euler characteristic of the space $C_I$ is equal to the Euler characteristic of the set of fixed points of the action. One can see that, if $k < t$, the only fixed point of this action is the $m \times n$-matrix
$$\left( \begin{array}{cc} \frac{1}{k} I_k & 0 \\0 & 0 \end{array} \right),$$
if $k \geq t$, the action has no fixed points. Therefore
$$\chi(C_I) = \left\{ \begin{array}{cl} 1 & \mbox{for } k < t, \\
0 & \mbox{for } k \geq t. \end{array} \right.$$
Summarizing we have
$$\overline{\chi}(L_{m,n}^t) = \sum_{k=1}^{t-1} (-1)^{k-1} \binom{m}{k}-1 = \sum_{k=0}^{t-1} (-1)^{k-1} \binom{m}{k} = (-1)^t \binom{m-1}{t-1}.$$
\end{proof}

Following \cite{EGGD}, for $i \leq j$, let
\begin{eqnarray*}
n_{ij} & := & \indrad (d\ell; M_{m-i+1,n-i+1}^{j-i+1},0) = (-1)^{d_{ij}-1} \overline{\chi}(L_{m-i+1,n-i+1}^{j-i+1}) \\ & = & (-1)^{(m+n)(j-i)} \binom{m-i}{m-j},
\end{eqnarray*}
where $d_{ij}$ is the dimension of $M_{m-i+1,n-i+1}^{j-i+1}$ equal to $(m-i+1)(n-i+1)-(m-j+1)(n-j+1)$.

\begin{proposition} \label{Propnit}
One has
$$\indrad (\omega;X,0) = \sum_{i=1}^t n_{it} \indPHN (\omega; X_i,0) + (-1)^{\dim X -1} \overline{\chi}(X,0).$$
\end{proposition}

\begin{proof} Consider the 1-form $\omega$ on the essential smoothing $\widetilde{X}$ of the singularity $(X,0)$. There exists a 1-form $\widetilde{\omega}$ which coincides with $\omega$ in a neighbourhood of the sphere $S_\eps$ such that in a neighbourhood of each singular point $x_0$ on $\widetilde{X}_i \setminus \widetilde{X}_{i-1}$ it looks as follows. There exists a local biholomorphism $h : (\CC^{\dim \widetilde{X}_i} \times M_{m-i+1,n-i+1},0) \to (\CC^N,x_0)$ which sends $\CC^{\dim \widetilde{X}_i} \times M_{m-i+1,n-i+1}^{t-i+1}$ onto $\widetilde{X}$ (and therefore $\CC^{\dim \widetilde{X}_i} \times \{ 0\}$ onto $\widetilde{X}_i$) and one has $h^\ast \widetilde{\omega} = p_1^\ast \omega' + p_2^\ast d \ell$ where $p_1$ and $p_2$ are the projections of $\CC^{\dim \widetilde{X}_i} \times M_{m-i+1,n-i+1}$ to the first and second factor respectively, $\omega'$ is a 1-form on $(\CC^{\dim \widetilde{X}_i},0)$ with a non-degenerate zero at the origin (and therefore $\indrad (\omega';\CC^{\dim \widetilde{X}_i},0)= \pm 1$),
and $\ell$ is a generic linear form on $M_{m-i+1,n-i+1}$. Note that the local form of the singular points of the 1-form $\widetilde{\omega}$ here is different from that in the proof of Proposition~\ref{PropPH}. The sum of the indices of the zeros of the 1-form $\widetilde{\omega}$ on $\widetilde{X}_i \setminus \widetilde{X}_{i-1}$ coincides with the PHN-index $\indPHN (\omega;X_i,0)$ of the 1-form $\omega$ on the variety $X_i$. Therefore the sum of the radial indices of the singular points of the 1-form $\widetilde{\omega}$ on $\widetilde{X}$ is equal to
$$\sum_{i=1}^t \indPHN (\omega;X_i,0) \cdot \indrad (d \ell; M_{m-i+1,n-i+1}^{t-i+1},0) = \sum_{i=1}^t \indPHN (\omega;X_i,0) n_{it}.$$
On the other hand it is equal to
$$\indrad (\omega;X,0) + (-1)^{\dim X}(\chi(\widetilde{X})-1) = \indrad (\omega;X,0) +(-1)^{\dim X} \overline{\chi}(X,0).$$
\end{proof}

Let $\cal N$ be the upper triangular $t \times t$-matrix with the entries $n_{ij}$ for $i \leq j \leq t$, and let ${\cal M}=(m_{ij})= {\cal N}^{-1}$. One can see that, for $i \leq j \leq t$,
$$m_{ij} = (-1)^{(m+n+1)(j-i)} \binom{m-i}{m-j}.$$
(This follows from the known formula $\sum\limits_{k=r}^s (-1)^k \binom{s}{k} \binom{k}{r}=0$ for $0 \leq r <s$, see, e.g., \cite[(2.1.4)]{Ha}.)

The inverse to Proposition~\ref{Propnit} is the following statement.

\begin{proposition} \label{PropPHN}
One has
$$\indPHN (\omega;X,0) = \sum_{i=1}^t m_{it} \left( \indrad (\omega; X_i,0) + (-1)^{\dim X_i} \overline{\chi}(X_i,0) \right).$$
\end{proposition}

\section{Isolated determinantal singularities}
For isolated determinantal singularities, the relations between the PH-, the PHN- and the radial indices simplify but are somewhat different in the (determinantally) smoothable and in the non-smoothable cases (i.e.\ for $N < (m-t+2)(n-t+2)$ and $N=(m-t+2)(n-t+2)$ respectively).

For isolated smoothable singularities all Poincar\'e--Hopf indices (including the Poincar\'e--Hopf-Nash index) coincide and they are equal to
$$\indPHs (\omega; X,0) =  \indrad (\omega; X,0) + (-1)^{\dim X} \overline{\chi}(X,0).$$

Let $(X,0)=F^{-1}(\Mmn^t) \subset (\CC^N,0)$ be an isolated non-smoothable determinantal singularity, i.e.\ $N=(m-t+2)(n-t+2)$. The singular locus $\widetilde{X}_{t-1}= \widetilde{F}^{-1}(\Mmn^{t-1})$ of an essential smoothing $\widetilde{X}= \widetilde{F}^{-1}(\Mmn^t)$ of the singularity $(X,0)$ consists of isolated points. The number of these points is the Euler characteristic $\chi(X_{t-1},0)$ of the stratum $\widetilde{X}_{t-1}$. Therefore the relations from Propositions~\ref{PropPH} and \ref{PropPHN} reduce to the following ones.

\begin{proposition}
For $N=(m-t+2)(n-t+2)$, the relation between the PH- and the radial index reduces to
$$
\indPHk (\omega;X,0)= \indrad (\omega;X,0) + (-1)^{\dim X} \left( \overline{\chi}(X,0)+ \chi(X_{t-1},0)\overline{\chi}(G_{t-1}^k) \right)
$$
where
$$\overline{\chi}(G_{t-1}^1)=n-t+1, \quad \overline{\chi}(G_{t-1}^2)=m-t+1, \quad \overline{\chi}(G_{t-1}^3)=(m-t+2)(n-t+2)-1$$
and $k=1,2,3$.
The relation between the PHN- and the radial index reduces to
\begin{eqnarray*}
\lefteqn{\indPHN (\omega;X,0)} \\
& = & \indrad (\omega;X,0) + (-1)^{\dim X} \overline{\chi}(X,0)+ (-1)^{m+n+1}(m-t+1) \chi(X_{t-1},0).
\end{eqnarray*}

\end{proposition}

One has the following formula for $\chi(X_{t-1},0)$ in this case.

\begin{proposition} Let $I_F^{t-1}$ be the ideal in the ring ${\cal O}_{\CC^N,0}$ generated by all the $(t-1) \times (t-1)$-minors of the matrix $F(x)$. Then
$$\chi(X_{t-1},0) = \dim {\cal O}_{\CC^N,0}/I_F^{t-1}.$$
\end{proposition}

\begin{proof} This follows from the fact that $\chi(X_{t-1},0) $ is the intersection number of the image $F(\CC^N)$ of the map $F$ with the Cohen-Macaulay variety $\Mmn^t \subset \Mmn$ defined by vanishing of the $(t-1) \times (t-1)$-minors.
\end{proof}

\section{The homological index for IDS}
For an isolated singular point of a holomorphic 1-form $\omega$ on an isolated singularity $(X,0)\subset (\CC^N,0)$ (of pure dimension $d$) the homological index was defined in \cite{EGS}. The homological index $\indhom\omega=\indhom(\omega; X,0)$ is $(-1)^d$ times the Euler characteristic of the complex
$$
0 \to \Omega^0_{X,0}
\stackrel{\wedge \omega}{\to} \Omega^1_{X,0} \stackrel{\wedge \omega}{\to} \ldots \stackrel{\wedge \omega}{\to} \Omega^d_{X,0} \to 0\,, 
$$
where $\Omega_{X,0}^k$ is the module of holomorphic differential $k$-forms on $(X,0)$, i.e. the module $\Omega_{\CC^N,0}^k$ factorized by the equations of variety $X$ and by the wedge products of their differentials with the module $\Omega_{\CC^N,0}^{k-1}$. For an isolated complete intersection singularity
$(X,0)=\{f_1=\ldots=f_{k}=0\}\subset (\CC^N,0)$ ($k:=N-d$) the homological index $\indhom\omega$ is equal to $\dim_\CC \Omega_{X,0}^d/\omega\wedge \Omega_{X,0}^{d-1}$, it coincides with the PH- (or GSV-) index, and it is equal to the dimension of the factor ring $\calA_{X,\omega}$ of the ring $\calO_{\CC^N,0}$
of germs of functions on $(\CC^N, 0)$ by the ideal generated by the functions $f_i$ and the $(k+1)\times(k+1)$-minors of the matrix
$$
\left( \begin{array}{ccc} \frac{\partial f_1}{\partial x_1} & \cdots &
\frac{\partial f_1}{\partial x_N} \\
\vdots & \ddots & \vdots \\
\frac{\partial f_k}{\partial x_1} & \cdots & \frac{\partial f_k}{\partial
x_N}\\
A_1 & \cdots & A_N
\end{array} \right).
$$

\begin{proposition}
For a holomorphic 1-form $\omega$ on an IDS $(X,0)$ with an isolated singular point (zero) at the origin one has
$$
\indhom\omega =\dim_\CC \Omega_{X,0}^d/\omega\wedge \Omega_{X,0}^{d-1}\,.
$$
\end{proposition}

\begin{proof}
We use the following version of a particular case of the de Rham Lemma from \cite[Lemma 1.6]{Gr}.

\begin{statement}
Let $\omega$ be a holomorphic 1-form on an isolated singularity $(X,0)$.
Then the kernel of the mapping
$$
\wedge\omega:\Omega_{X,0}^p\to \Omega_{X,0}^{p+1}
$$
is equal to $\omega\wedge\Omega_{X,0}^{p-1}$ if $0\le p < \mbox{\rm codh\,}\calO_{X,0}-\dim C(\omega, X)$, where $\mbox{\rm codh\,}\calO_{X,0}$ is the homological codimension of the ring $\calO_{X,0}$ and $C(\omega, X)$ is the set of zeros of the 1-form $\omega$ on $X$.
\end{statement}

In \cite{Gr} the statement is formulated for $\omega=df$ for a holomorphic function $f:(\CC^N,0)\to(\CC,0)$. However (as  G.-M.~Greuel pointed out to us) the proof is also valid for an arbitrary holomorphic 1-form $\omega$.

In our case the 1-form $\omega$ has an isolated singular point on $(X,0)$ and therefore $C(\omega, X)=\{0\}$. Moreover $(X,0)$ (being a determinantal singularity) is Cohen-Macaulay \cite{HE}. Therefore $\mbox{\rm codh\,}\calO_{X,0} =\dim X$.
\end{proof}

For a reduced space curve singularity (which is always determinantal of type $(m, m+1, m)$) V.V.~Goryunov \cite{Gor} has defined a Milnor number for a holomorphic function $f$ with an isolated singularity on the curve. This notion was extended to arbitrary reduced curve singularities by D.~Mond and D.~van~Straten \cite{MvS}. Almost the same definition applies to a holomorphic 1-form with an isolated singular point on the curve giving a number which can be considered as an index (GMvS-index $\indGMvS\omega$) of the 1-form $\omega$. If the curve singularity is smoothable, the GMvS-index coincides with the Poincar\'e--Hopf index defined by a smoothing. (All smoothings of a curve singularity have the same Euler characteristic and therefore the Poincar\'e--Hopf index is independent of the smoothing.)

\begin{sloppypar}

The homological index of a 1-form on a curve singularity $(C,0)$ is, in general, different from the GMvS-index. Let $(\overline C, \overline 0)$ be the normalization of the curve singularity $(C,0)$ and let $\tau=\dim_\CC\ker (\Omega_{C,0}^1\to \Omega_{\overline C, \overline 0}^1)$, $\lambda=\dim_\CC \omega_{C,0}/c(\Omega_{C,0}^1)$, where $\omega_{C,0}$ is the dualizing module of Grothendieck, $c:\Omega_{C,0}^1\to \omega_{C,0}$ is the class map. Then one has
$$
\indhom\omega=\indGMvS\omega -\lambda+\tau
$$
(see, e.g., \cite{EGS}).

\end{sloppypar}

By \cite{MvS}, the GMvS-index of a 1-form $\omega=\sum\limits_{i=1}^3 A_idx_i$ on a space curve singularity $(C,0)$ defined by vanishing of the $m\times m$-minors $\Delta_1$, \dots, $\Delta_{m+1}$ of an $m\times (m+1)$-matrix $F(x)$ is equal to the dimension of the factor algebra $\calA_{C,\omega}$ of the algebra $\calO_{\CC^3,0}$ by the ideal generated by the equations $\Delta_1$, \dots, $\Delta_{m+1}$ of the curve and by the $3\times 3$-minors of the matrix
$$
\left( \begin{array}{ccc} \frac{\partial \Delta_1}{\partial x_1} &   \frac{\partial \Delta_1}{\partial x_2}&
 \frac{\partial \Delta_1}{\partial x_3} \\
\vdots & \vdots & \vdots \\
\frac{\partial \Delta_{m+1}}{\partial x_1} &   \frac{\partial \Delta_{m+1}}{\partial x_2}&
 \frac{\partial \Delta_{m+1}}{\partial x_3} \\
A_1 & A_2 & A_3
\end{array} \right) =: \left( \begin{array}{c} d\Delta_i \\ \omega \end{array} \right).
$$
In \cite{MvS} the statement is formulated for a 1-form $\omega$ being the differential of a holomorphic function, however, the proof can be adapted to the general case. In fact the statement for arbitrary 1-forms can be deduced from the one for differentials of functions.

A natural generalization of the algebras $\calA_{X,\omega}$ for an ICIS $(X,0)$ and $\calA_{C,\omega}$ for a space curve singularity 
$(C,0)$ to a 1-form $\omega=\sum\limits_{i=1}^N A_idx_i$ on an IDS $(X,0)=F^{-1}(\Mmn^t)\subset(\CC^N, 0)$ of type $(m,n,t)$ is the algebra $\calA_{X,\omega}$ 
defined as the factor algebra of the algebra $\calO_{\CC^N,0}$ by the ideal $I_{X,\omega}$ generated by the 
$t \times t$-minors $m_{IJ}^t$ of the matrix $F(x)$ and by the $(c+1)\times(c+1)$-minors of the matrix
$$
\left( \begin{array}{c}  dm_{IJ}^t \\ \omega \end{array} \right)
$$ 
whose rows are the components of the differentials of the minors $m_{IJ}^t$ ($\# I=\# J=t$) and of the 1-form $\omega$, where $c= (m-t+1)(n-t+1)$ is the codimension of the IDS $(X,0)$. The algebra $\calA_{X,\omega}$ has finite dimension as a complex vector space.

\begin{proposition} \label{PropIneq}
For a (determinantally) smoothable IDS $(X,0)$ one has
$$
\dim_\CC \Omega_{X,0}^d/\omega\wedge \Omega_{X,0}^{d-1}\ge \indPHs \omega\,,
$$
$$
\dim_\CC \calA_{X,\omega}\ge \indPHs \omega\,.
$$
\end{proposition}

\begin{proof}
Let $F_\lambda:U\to \Mmn$ be a 1-parameter deformation of the map $F$ ($F_0=F$) such that $X_\lambda=F_\lambda^{-1}(\Mmn^t)$ is a smoothing of the IDS $(X,0)$, and let $\omega_\lambda$ be a 
deformation of the 1-form $\omega$ such that, for $\lambda\in\CC\setminus\{0\}$, the 1-form $\omega_\lambda$ has only non-degenerate zeros on the manifold $X_\lambda$ (the number of these points is the PH-index $\indPHs \omega$ of the 1-form $\omega$). The family $\omega_\lambda$ defines a 1-form $\check{\omega}$ on $Y=U\times \CC$ (without the differential $d\lambda$). The $\calO_{U\times \CC}$-modules 
$$
\Omega_{Y/\CC}^d/\langle m_{IJ} \Omega_{Y/\CC}^d, dm_{IJ}\wedge\Omega_{Y/\CC}^{d-1}, \check{\omega}\wedge\Omega_{Y/\CC}^{d-1} \rangle \quad \mbox{and} \quad
\calO_Y/I_{Y, \check{\omega}}
$$
($\Omega_{Y/\CC}^d:=\Omega_{Y}^d/d\lambda \wedge \Omega_{Y}^{d-1}$) have support on the curve   consisting of singular points of the 1-forms $\omega_\lambda$ on the fibres $X_\lambda$. For each of these sheaves $\calF$ and for 
$y \in Y$, let $\nu(y)$ be the dimension of the vector space $\CC \otimes_{{\cal O}_{\CC, \lambda}} {\cal F}_y$. For $\lambda \neq 0$, $\nu(\lambda)= \sum_{x \in U} \nu(x,\lambda)$ is equal to the PH-index ${\rm ind}_{\rm PH}\, \omega$ of the 1-form $\omega$. Now the statement follows from the semicontinuity of the function $\nu(\lambda)$.
\end{proof}

One can show that, in general, the inequalities in Proposition~\ref{PropIneq} may be strict and the left hand sides of them may also be different.

\begin{example} Let $(X,0) \subset (\CC^4,0)$ be the surface determinantal singularity of type $(2,3,2)$ defined by the matrix
$$\left( \begin{array}{ccc} z & y+u & x \\
u & x & y \end{array} \right)$$
($x,y,z,u$ are the coordinates in $\CC^4$), and let $\omega = du$. Then one has
$${\rm ind}_{\rm PH}\, \omega =3.$$
This follows from the fact that $X$ is the space of a 1-parameter deformation (with parameter $u$) of the space curve singularity $(C,0)$ defined by the matrix
$$\left( \begin{array}{ccc} z & y & x \\
0 & x & y \end{array} \right).$$
A versal deformation of the singularity $(C,0)$ has the form
$$\left( \begin{array}{ccc} z & y+b & x+c \\
a & x & y \end{array} \right)$$
with the discriminant $a(b^2-c^2)=0$ (see, e.g., \cite{FK}). A smoothing of $X$ can be given by taking $a=b=u$, $c=\lambda \neq 0$. The corresponding 1-parameter family in $u$ intersects the discriminant at 3 points where the function $u$ has non-degenerate critical points.

On the other hand
\begin{eqnarray*}
\dim_\CC \Omega^2_{X,0}/\omega \wedge \Omega^1_{X,0} & = & 6, \\
\dim_\CC {\cal A}_{X,\omega} & = & 5.
\end{eqnarray*}
(The calculations were made with the help of the computer algebra system {\sc Singular} \cite{Si}.)
\end{example}

\begin{remark} For a 1-form on a smoothable IDS \ $(X,0) \subset (\CC^N,0)$ with an isolated singular point on $(X,0)$, there is a well defined quadratic form on the module $\Omega^d_{X,0}/\omega \wedge \Omega^{d-1}_{X,0}$ constructed in a way similar to \cite{EGMZ}.
\end{remark}

\bigskip
\noindent Leibniz Universit\"{a}t Hannover, Institut f\"{u}r Algebraische Geometrie,\\
Postfach 6009, D-30060 Hannover, Germany \\
E-mail: ebeling@math.uni-hannover.de\\

\medskip
\noindent Moscow State University, Faculty of Mechanics and Mathematics,\\
Moscow, GSP-1, 119991, Russia\\
E-mail: sabir@mccme.ru
\end{document}